\documentclass[12pt]{amsart}
\usepackage{amsmath,amsthm,amsfonts,amssymb,amscd,amstext}
\usepackage{blindtext}
\usepackage{inputenc}
\usepackage[dvips]{graphicx}
\usepackage{psfrag}
\usepackage{euscript}
\usepackage{a4wide}

\numberwithin{equation}{section}

\usepackage[colorlinks=true,linkcolor=red,urlcolor=black]{hyperref}

\theoremstyle{plain}
\newtheorem{maintheorem}{Theorem}

\newtheorem{theorem}{Theorem}[section]

\newtheorem{corollary}[theorem]{Corollary}
\newtheorem{lemma}[theorem]{Lemma}

\newtheorem{problem}{Problem}
\theoremstyle{definition}

\newtheorem{definition}{Definition}

\renewcommand{\epsilon}{\varepsilon}

\newcommand{\spp}{\operatorname{sp}}

\begin{document}

\title[A note on basis problem in normed spaces]
{A note on basis problem in normed spaces}


\thanks{V.C. was supported by CAPES, L.S. is partially supported by FAPERJ-Funda\c{}c\~ao Carlos Chagas Filho de Amparo \`a Pesquisa do Estado do Rio de Janeiro Projects APQ1-E-26/211.690/2021 SEI-260003/015270/2021 and JCNE-E-26/200.271/2023 SEI-260003/000640/2023, by CAPES — Finance Code 001, CNPq Projeto Universal 404943/2023-3. 
V.C., J.R. and L.S. are grateful to Instituto de Matem\'atica da Universidade Federal da Bahia for the hospitality in preparation of this paper. {{\bf \textup{2010} Mathematics Subject Classification} Primary: 46B15; Secondary: 46B20, 46B28. Keywords: {Normed spaces, Schauder basis, Spectral theorem}}}


\author{Vinicius Coelho}

\author{Joilson Ribeiro}

\author{Luciana Salgado}

\address[V.C.]{Universidade Federal do Oeste da Bahia, Centro Multidisciplinar de Bom Jesus da Lapa\\
Av. Manoel Novais, 1064, Centro, 47600-000 - Bom Jesus da Lapa-BA-Brazil}
\email{viniciuscs@ufob.edu.br}

\address[J.R.]{Universidade Federal da Bahia,
Instituto de Matem\'atica e Estat\'\i stica\\
Av. Adhemar de Barros, S/N , Ondina,
40170-110 - Salvador-BA-Brazil}
\email{joilsonor@ufba.br}

\address[L.S.]{Universidade Federal do Rio de Janeiro, Instituto de
   Matem\'atica\\
   Avenida Athos da Silveira Ramos 149 Cidade Universit\'aria, P.O. Box 68530,
   21941-909 Rio de Janeiro-RJ-Brazil }
 \email{lsalgado@im.ufrj.br, lucianasalgado@ufrj.br}

\date{\today}
\maketitle

\begin{abstract}

In this note, we give a proof of the well known criterion of Banach-Grunblum and the Bessaga-Pe\l{}czy\'nski Theorem in normed spaces context, not necessarily complete (Banach) one. As application of these results, we show the Principle of Selection of Bessaga-Pe\l{}czy\'nski for normed spaces and the Spectral Theorem for compact self-adjoint operators on inner product spaces.

\end{abstract}

\section{Introduction}\label{sec:intro}

In 1933, S. Banach affirmed that \emph{every infinitely dimensional Banach space contains an infinite dimensional subspace with a basis}, without proof. In 1958, Bessaga and Pe\l{}czy\'nski \cite{bespel58} developed several generalizations and modifications of Banach's claim, and proved their well known Selection Principle in this setting.

In this paper, as a short complement of this theory, we work on Banach's problem for normed (not necessarily complete) spaces, giving an expected generalization of the previous results, but not formally proved yet, as long as we know. We also generalize Banach-Grublum's criterion in this context, and give a Spectral Theorem on inner product spaces.

Vector space basis play a key role in the most varied problems of functional analysis, as example we refer the reader to see \cite{bespel58,Mc72,Sch1927,Sch1928,VK2022,Yar2022}.

For any vector space, it is well known that there exists an algebraic basis (or Hamel basis). However, there is another notion of basis, due to J. Schauder \cite{Sch1927,Sch1928}, defined as follows:

\begin{definition}
A sequence $(x_{n})_{n=1}^{\infty}$ in a normed space $X$ is called a {\it Schauder basis} for $X$ if for each $x$ in $X$ there is an unique sequence $(a_{n})$ of scalars such that $x = \sum \limits_{i=1}^{\infty}a_{n}x_{n}$.
\end{definition}

The uniqueness of the representation allows us to consider the linear operator for each $n$ in $\mathbb{N}$:

\begin{center}
$x_{n}^{*}: X \to \mathbb{K}\text{, } x_{n}^{*}\left(\sum \limits_{i=1}^{\infty}a_{j}x_{j}\right) =a_{n}$,
\end{center}
this operators are called  \textit{coefficient operators} (or \textit{coordinates operator}).

 Let $(x_{n})_{n=1}^{\infty}$ be a Schauder basis in the normed space $(X, \Vert \cdot \Vert)$, and consider the linear space $\mathcal{L}_{X} =\{ (a_{n})_{n=1}^{\infty} $ $|$  $ \sum \limits_{i=1}^{\infty}a_{n}x_{n}$ is convergent$\}$.  A computation shows that the function $\eta:  \mathcal{L}_{X} \to \mathbb{R}$ given by $\eta((a_{n})_{n=1}^{\infty}) := \sup \left\lbrace\left\Vert \sum \limits_{i=1}^{n} a_{i}x_{i} \right\Vert : n \in \mathbb{N}  \right\rbrace $ gives a norm in $\mathcal{L}_{X}$.

In this paper, our purpose is to provide the criterion of Banach-Grunblum and the Bessaga-Pe\l{}czy\'nski Theorem for normed spaces, generalizing the Banach's version \cite{Banach32,besspel58}. In the conclusion, we explain possible applications of this theory to pure and applied mathematical sciences, as the combination of least-squares method and a Schauder basis to provide a numerical solution for a wide class of linear differential or integral equations (see \cite{PR2005}).

To deduce these results, we start with the following broad notion of Schauder basis.

\begin{definition}
Let $X$ be a normed space and $(x_{n})_{n=1}^{\infty}$  be a Schauder basis in $X$. We say that $(x_{i})_{i=1}^{\infty}$ is an \textit{essential Schauder basis} for $X$ if $T_{X}: \mathcal{L}_{X} \to X$ given by $T_{X}((a_{n})_{n=1}^{\infty}) = \sum \limits_{n=1}^{\infty} a_{n}x_{n}$ is an isomorphism.
\end{definition}

It is not difficult to show that in Banach spaces, every Schauder basis is an essential Schauder basis. This identification is important because, with it, any space that admits a Schauder basis can be seen as a space of sequences. In the original definition of Schauder, there was the requirement that coordinates functional should be continuous. However, in 1932 Banach \cite[pag 111]{Banach32} showed that in complete normed spaces this is always true. But, if the space is not complete, this assertion is false (see \cite[Example 12.5]{Sw10}).

A sequence $(x_{n})_{n=1}^{\infty}$ may not be a Schauder basis for a normed space $X$, because $\overline{[x_{n}: n \in \mathbb{N}]}$ does not reach all the space $X$. In this case, we say that:

\begin{definition}
A sequence $(x_{n})_{n=1}^{\infty}$ in a normed space $X$ is a  \textit{basic sequence} if the sequence $(x_{n})_{n}$  is a  Schauder basis for $\overline{[x_{n}: n \in \mathbb{N}]}$.
\end{definition}

In Banach spaces theory, we have the following practical and useful criterion for deciding whether a given sequence is basic or not.

\begin{theorem} \label{B-GforBanach}(Banach-Grunblum's Criterion)
A sequence $(x_{n})_{n=1}^{\infty}$ of non null vector in a Banach space $X$ is a basic sequence if, and only if, there exists $M \geqslant 1$  such that for all  sequence of scalar $(a_{n})_{n=1}^{\infty}$:

\begin{align}
\left\Vert \sum \limits_{i=1}^{m} a_{i}x_{i}  \right\Vert \leqslant M \left\Vert \sum \limits_{i=1}^{n} a_{i}x_{i}  \right\Vert
\end{align} whenever $n \geqslant m$.
\end{theorem}
\begin{proof}
See e.g. \cite[Proposition 1.1.9]{AlKa06} or \cite[Theorem 10.3.13]{bop}.
\end{proof}

One of the consequences of the above theorem was proved by S. Banach in 1932 and another one, by M. Grunblum in 1941. We generalize the Theorem \ref{B-GforBanach} for normed spaces setting. For this purpose, it is necessary to introduce the concept of essential basic sequence.

\begin{definition}
A sequence $(x_{n})_{n=1}^{\infty}$ in a normed space $X$ is called an \textit{essential basic sequence} if it is an essential Schauder basis for $\overline{[x_{n}: n \in \mathbb{N}]}$.
\end{definition}

Note that if $X$ is a Banach space, then every basic sequence is an essential basic sequence. 

Let $(S,\| \cdot \|)$ be a normed space. We denote by $\widehat{S}$ the completion of $S$ such that $S$ is dense in $\widehat{S}$.

\begin{maintheorem} \label{B-GforNormed} (Banach-Grunblum's criterion for normed spaces)
Let $(x_{n})_{n=1}^{\infty}$ be a sequence of non null vectors in a normed space $X$. Then the following conditions are equivalents.
\begin{itemize}
\item[$(i)$] $(x_{n})_{n=1}^{\infty}$ is an essential  Schauder basis for $\widehat{[x_{n}: n \in \mathbb{N}]} \subseteq \widehat{X}$;
\item[$(ii)$]  $(x_{n})_{n=1}^{\infty}$ is an essential basic sequence in  $ X$;
\item[$(iii)$] there exists $M \geqslant 1$  such that for all sequence of scalar $(a_{n})_{n=1}^{\infty}$:

\begin{align}\label{eq12}
\left\Vert \sum \limits_{i=1}^{m} a_{i}x_{i}  \right\Vert \leqslant M \left\Vert \sum \limits_{i=1}^{n} a_{i}x_{i} \right\Vert
\end{align} whenever $n \geqslant m$.
\end{itemize}
\end{maintheorem}

As it has been said before, in his classic book \cite{Banach32}, Banach announced, without proof, that in every complete normed (Banach) space of infinite dimension there is an infinite dimensional subspace with Schauder basis. The proof of this result only appeared in the literature in 1958, in a celebrated article by Bessaga and Pe\l{}czy\'nski \cite{bespel58} (in the same year Bernard R. Gelbaum \cite{G58} also presented another proof). The demonstration presented in \cite{bespel58} is a consequence of the main result of their work, which became known as \emph{selection's principle of Bessaga-Pe\l{}czy\'nski}.

Here, we show the Bessaga-Pe\l{}czy\'nski's Theorem for normed spaces, and so the Selection Principle for normed spaces, as an application. 

We introduce the definition of \emph{equivalent sequence} for Essential Schauder's Basis.

\begin{definition}
Let $(x_{n})_{n=1}^{\infty}$ be an essential Schauder's basis in a normed space $X$ and $(y_{n})_{n=1}^{\infty}$ be an essential Schauder's basis in a normed space $Y$. We say that $(x_{n})_{n=1}^{\infty}$ is \textit{equivalent} to $(y_{n})_{n=1}^{\infty}$. In this case, we write  $(x_{n})_{n=1}^{\infty}\approx (y_{n})_{n=1}^{\infty}$, to means that, for any scalar's sequence, the series $\sum \limits_{n=1}^{\infty} a_{n}x_{n}$ is convergent in $\overline{[x_{n}: n \in \mathbb{N}]} \subseteq X$ if, and only if, the series $\sum \limits_{n=1}^{\infty} a_{n}y_{n}$ is convergent in $\overline{[y_{n}: n \in \mathbb{N}]} \subseteq Y$.
\end{definition}

We are able to present the Bessaga-Pe\l{}czy\'nski Theorem for normed spaces.

\begin{maintheorem}\label{th:1}
Let $X$ be a normed space, $(x_{n})_{n=1}^{\infty}$ be an essential basic sequence in $X$, and $(x_{n}^{*})_{n=1}^{\infty}$ be the functional coefficients. If  $(y_{n})_{n=1}^{\infty}$ is a sequence in $X$ such that

\begin{align}
0<\sum \limits_{n=1}^{\infty} \| x_{n} - y_{n} \|. \| x_{n}^{*} \|=: \lambda < 1
\end{align}

then $(y_{n})_{n=1}^{\infty}$ is an essential basic sequence in $X$ equivalent  to $(x_{n})_{n=1}^{\infty}$.
\end{maintheorem}

\section{Applications}\label{sec:app}

\subsection{Bessaga-Pe\l{}czy\'nski's Selection Principle}

Using Theorem \ref{B-GforNormed}  and  \ref{th:1}, we obtain the following Bessaga-Pe\l{}czy\'nski's Selection Principle and its consequence for normed spaces, as follow.

\begin{definition}
Let $(x_{n})_{n=1}^{\infty}$ be an essential Schauder's basis in  a normed space $X$, and $(k_{n})_{n=0}^{\infty}$ be a sequence strictly increasing of  positive integers, with $k_{0}=0$. A sequence of non null vectors $(y_{n})_{n=1}^{\infty}$ in $X$ is called \textit{essential block basic sequence relative to $(x_{n})_{n=1}^{\infty}$ } if

\begin{center}
$y_{n} = \sum \limits_{i=k_{n-1}+1}^{k_{n}} b_{i}x_{i}$
\end{center}

where  $b_{i} \in \mathbb{K}$.
\end{definition}

The Bessaga-Pe\l{}czy\'nski's Selection Principle is the following.

\begin{theorem}\label{mth:2} Let $(x_{n})_{n=1}^{\infty}$ be an essential Schauder basis in a  normed space $X$ and $(x_{n}^{*})_{n=1}^{\infty}$ be the functional coefficients. If $(y_{n})_{n=1}^{\infty}$ is a sequence in $X$ such that $\inf \limits_{n} \|y_{n}\| > 0$ and

\begin{center}
$\lim \limits_{n\to \infty} x_{i}^{*}(y_{n}) = 0$ for all $i \in \mathbb{N}$
\end{center}

then $(y_{n})_{n=1}^{\infty}$ contains  an essential basic subsequence equivalent to block basic sequence relative  to $(x_{n})_{n=1}^{\infty}$.
\end{theorem}

\begin{corollary}\label{mcor:1}
Let $X$ be a normed space, $(y_{n})_{n=1}^{\infty}$ is a sequence in $X$ such that $\inf \limits_{n} \|y_{n}\| > 0$ and $y_{n} \to 0$ weakly. Then   $(y_{n})_{n=1}^{\infty}$ contains an essential basic subsequence.
\end{corollary}

\subsection{Banach problem for normed spaces}

As mentioned before, the Bessaga and Pe\l{}czy\'nski's proof that in every Banach space of infinite dimension there is an infinite-dimensional subspace with Schauder basis is a consequence of the nowadays known as the Bessaga-Pe\l{}czy\'nski selection principle. As we have obtained this result for normed spaces, it is possible to follow the same steps and to show the same result in this broad setting. We stress that this is already known result (see \cite{Day62}). But here, the proof is done in an elementary way, by using only the Banach-Grunblum's Criterion for normed spaces (Theorem \ref{B-GforNormed}).

\begin{theorem}\label{mthm:Banach-Mazur}
Every normed space contains an infinite-dimensional closed subspace with Schauder basis in which the canonical projections $(P_{n})_{n}$ are bounded operators. Moreover, $\sup \limits_{n} \|P_{n}\| < + \infty $.
\end{theorem}

\subsection{Spectral Theorem on inner product spaces}
Let $T: N \to N$ be a compact self-adjoint operator, and $\widehat{T}: \widehat{N} \to \widehat{N}$ be the  compact self-adjoint operator such that  $\widehat{T}$ is the bounded linear extension of $T$.

The following is the Spectral Theorem on inner product spaces.

\begin{theorem}\label{sptns}
Let $N$ be  an inner product space, $T: N \to N$ be a compact self-adjoint operator such that $T \neq 0$,  $\spp (T)$ be the spectrum of $T$, and $\spp (\widehat{T})$ be the spectrum of $\widehat{T}$.
\begin{itemize}
\item[$(i)$] If  $\spp (T)$ is an infinite set such that $\spp (T) \subsetneq \spp (\widehat{T})$, then for each $x \in N$, there exists $w_{x} \in N $ such that  $T(x) =  \sum \limits_{i=1}^{\infty} x_{i} \lambda_{i} \langle x, x_{i} \rangle + w_{x}$;
\item[$(ii)$] If  $\spp (T)$ is an infinite set such that $\spp (T) =\spp (\widehat{T})$,   then for each $x \in N$,  $T(x) =  \sum \limits_{i=1}^{\infty} x_{i} \lambda_{i} \langle x, x_{i} \rangle$;
\item[$(iii)$] If the cardinality of $\spp (T)$  is a non-null natural number,  then there exists a natural number $s \in \mathbb{N} \setminus \{0\}$ such that for all $x \in N$,  $T(x) =  \sum \limits_{i=1}^{s} x_{i} \lambda_{i} \langle x, x_{i} \rangle$

\end{itemize}

where $(\lambda_{i})'s$ are the eigenvalues  of  $T$.

\end{theorem}

\section{Problems}\label{sec:prob}

It is known that every Banach space with Schauder basis is separable (see \cite[Proposition 4.1.10]{Me98}). Banach has questioned if the converse holds: Does every separable Banach space admits a Schauder basis? In \cite{enflo}, Enflo gave a negative answer to this question, exhibiting a separable Banach space that has no Schauder basis.

\begin{problem} (The basis problem for normed spaces)
Let $X$ be a separable normed space which is not a Banach one. Does $X$ admits an essential Schauder basis?
\end{problem}

In the sequence, we introduce the definition of (essential) unconditional basis.

\begin{definition}
Let $X$ be a Banach space, and $(x_{n})_{n=1}^{\infty}$ a Schauder basis in $X$. The basis $(x_{n})_{n=1}^{\infty}$ is an unconditional basis if, for each $x$ in $X$,  there exists a unique expansion of the form

 \begin{center}
$x= \sum \limits_{n=1}^{\infty} a_{n}x_{n}$
 \end{center}
where the sum converges unconditionally.
\end{definition}

Gowers and Maurey in \cite{GoMa} showed that there exists a Banach space that do not contain unconditional basis.

\begin{definition}
Let $X$ be a normed space, and $(x_{n})_{n=1}^{\infty}$ an essential  Schauder basis in $X$. The basis $(x_{n})_{n=1}^{\infty}$ is an essential unconditional basis if, for each $x$ in $X$,  there exists a unique expansion of the form

 \begin{center}
$x= \sum \limits_{n=1}^{\infty} a_{n}x_{n}$
 \end{center}
where the sum converges unconditionally.
\end{definition}

We end this section with the following question.

\begin{problem} (The unconditional basic sequence problem for normed spaces)
Let $X$ be a normed space which it is not a Banach space. Does $X$ admits an essential unconditional basis?
\end{problem}

\subsection{Organization of the text}
In Sections \ref{sec:intro}, \ref{sec:app} and \ref{sec:prob}, we provide preliminary definitions
in order to present the statements of the main results together with some applications and
problems. In Section \ref{sec:aux.results} we state some auxiliary results and prove some useful
properties of Essential Schauder Basis. In Section \ref{sec:papp}, we give the proofs of applications of main results, divided
into three subsections \ref{sec:papp.1}, \ref{sec:papp.2} and \ref{sec:papp.3}, one for each of the applcations.  
In Section \ref{sec:main.result}, we give the proofs of our theorems, divided into two subsections \ref{sec:main.result.A} and 
\ref{sec:main.result.B}, one for each of the Main Theorems \ref{B-GforNormed} and \ref{th:1}, respectively.

\section{Auxilar Results}\label{sec:aux.results}

\subsection{Essential Schauder basis}

In the original Schauder's definition, there was the requirement that coordinates functional should be continuous. However, in \cite[pag 111]{Banach32} is proved that this condition holds for Banach spaces. But, if the space is not complete, this assertion is false (see \cite[Example 12.5]{Sw10}). We observe that if there exists an essential Schauder basis, then we recover this important property, as follows.

\begin{theorem}\label{fccontinuos}
Each coordinate functional associated to an essential Schauder basis $(x_{n})_{n=1}^{\infty}$  is a bounded linear application.
\end{theorem}
\begin{proof}
Fix $n \in \mathbb{N}$, and let $x \in X$ be an arbitrary element, so we can write $x =\sum \limits_{i=1}^{\infty}a_{j}x_{j}$. We are going to show that $x_{n}^{*}: X \to \mathbb{K}$ is a bounded linear application. In fact,

\begin{align*}
\|x_{n}\|\cdot |x_{n}^{*}(x)| =\|x_{n}x_{n}^{*}(x)\|  \\   
 = \Vert \sum \limits_{i=1}^{n}x^{*}_{i}(x)x_{i} - \sum \limits_{i=1}^{n-1}x^{*}_{i}(x)x_{i}\Vert\\
  \leq \Vert  \sum \limits_{i=1}^{n}x^{*}_{i}(x)x_{i} \Vert +   \Vert \sum \limits_{i=1}^{n-1}x^{*}_{i}(x)x_{i} \Vert\\ 
  \leq   2 \eta((a_{j})_{j=1}^{\infty})\leq 2  \Vert (T_{X})^{-1} \Vert \cdot \|x\|.
\end{align*}

\end{proof}

\begin{corollary}
Let $(x_{n})_{n=1}^{\infty}$ be an essential Schauder basis for the normed space $X$. Then, for each $n\in\mathbb{N}$, the linear operator

\begin{center}
$P_{n}: X \to X\text{, } P_{n}\left(\sum \limits_{i=1}^{\infty}a_{j}x_{j}\right) = \sum \limits_{i=1}^{n}a_{j}x_{j}$
\end{center}

is bounded.
\end{corollary}
\begin{proof}
Note that  $P_{n}(\cdot) = \sum \limits_{i=1}^{n} x^{*}_{i}(\cdot)x_{i}$, so by Theorem  (\ref{fccontinuos}), $P_{n}$ is a bounded operator.
\end{proof}

In the proof of the next result, we are going to use some arguments from the proof of Corollary 4.1.17 of \cite{Me98}.

\begin{theorem}\label{theorem11}
Let $(x_{n})_{n=1}^{\infty}$ be an essential Schauder basis for a normed space $X$ and $(P_{n})_{n}^{\infty}$ the canonical projections. Then $\sup \limits_{n} \| P_{n}\| < \infty$.
\end{theorem}
\begin{proof}
For each $x$ in $X$ such that $\|x\| \leq 1$, we have that

$ \left\Vert  P_{n}(x) \right\Vert =  \left\Vert  \sum \limits_{i=1}^{n}a_{j}x_{j} \right \Vert \leq \sup \limits_{n} \left\Vert \sum \limits_{i=1}^{n}a_{j}x_{j} \right \Vert =  \eta((a_{j})_{j=1}^{\infty}) \leq  \left\Vert (T_{X})^{-1} \right\Vert \cdot \|x\| \leq   \left\Vert (T_{X})^{-1} \right\Vert  $.

So, $\| P_{n}\| \leq  \|(T_{X})^{-1}\|$ for all $n$ in $\mathbb{N}$  and  $\sup \limits_{n} \| P_{n}\| \leq \|(T_{X})^{-1}\|$, and we are done.
\end{proof}

The number $K_{(x_{n})_{n=1}^{\infty}}:= \sup \limits_{n} \| P_{n}\|$  is called \textit{essential constant of basis} $(x_{n})_{n=1}^{\infty}$. Note that $\|P_{n}\| \geq 1$ for all $n$, then  $K_{(x_{n})_{n=1}^{\infty}} \geq 1$.

\begin{corollary}\label{cor11}
Let $(x_{n})_{n=1}^{\infty}$ be an essential Schauder basis for a normed space $X$ and $(x_{n}^*)_{n=1}^{\infty}$ be the coefficient operators. Then, for each $k \in \mathbb{N}$,

\begin{equation}
1 \leq \| x_{k}^{*}\| \cdot \|x_{k}\| \leq 2 K_{(x_{n})_{n=1}^{\infty}}
\end{equation}
\end{corollary}
\begin{proof}
First, we note that
\begin{equation}
1 = x_{k}^{*}(x_{k}) = |x_{k}^{*}(x_{k})| \leq  \|x_{k}^{*}\| \cdot \|x_{k}\|.
\end{equation}

Now, let $x$ be an arbitrary element of $X$ such that $x \neq 0$, so

\begin{align*}
\|x_{k}\|\cdot |x_{k}^{*}(x)| =\|x_{k}x_{k}^{*}(x)\| \\= \Vert \sum \limits_{i=1}^{k}x^{*}_{i}(x)x_{i} - \sum \limits_{i=1}^{k-1}x^{*}_{i}(x)x_{i} \Vert\\
  \leq  \Vert  \sum \limits_{i=1}^{k}x^{*}_{i}(x)x_{i}  \Vert +    \Vert  \sum \limits_{i=1}^{k-1}x^{*}_{i}(x)x_{i} \Vert\\ =  \Vert  P_{k}(x)  \Vert  +  \Vert P_{k-1}(x)  \Vert \\
\leq   \Vert  P_{k}  \Vert \cdot \|x\| +  \Vert  P_{k-1}(x)   \Vert \cdot\|x\| \\ \leq 2K_{(x_{n})_{n=1}^{\infty}}\|x\|.
\end{align*}

This implies that\\
$\|x_{k}\|\cdot \|x_{k}^{*}\| \leq  2K_{(x_{n})_{n=1}^{\infty}}$.

\end{proof}

\section{Proof of applications of main results}\label{sec:papp}

\subsection{Proof of Bessaga-Pe\l{}czy\'nski's Selection Principle}\label{sec:papp.1}

Using Theorem \ref{B-GforNormed}, it is possible to show the following result.

\begin{corollary}\label{cor22}
Let $(x_{n})_{n=1}^{\infty}$  be an essential basic sequence in a normed space $X$. Then $K_{(x_{n})_{n=1}^{\infty}} = \inf \{M:$ $M$ satisfies  (\ref{eq12})$\}$.

\end{corollary}

We are going to provide the definitions and required results to show the Bessaga-Pe\l{}czy\'nski's Selection Principle for normed spaces.

\begin{lemma}\label{lemma1}
Let $(x_{n})_{n=1}^{\infty}$ be an essential Schauder's basis in a normed space $X$ and $(y_{n})_{n=1}^{\infty}$ be an essential block basic sequence relative to $(x_{n})_{n=1}^{\infty}$. Then $(y_{n})_{n=1}^{\infty}$ is an essential basic sequence in $X$ and   $K_{(y_{n})_{n=1}^{\infty}}\leq  K_{(x_{n})_{n=1}^{\infty}}$.
\end{lemma}

\begin{proof}
Use the Banach-Grunblum's criterion for normed spaces, Theorem \ref{B-GforNormed}, and Corollary \ref{cor22}.
\end{proof}

Following the proof given in Theorem 4.3.19 of \cite{Me98} for Banach spaces, and using Corollary \ref{cor11}, Lemma \ref{lemma1} and Theorem \ref{th:1} we obtain the Bessaga-Pe\l{}czy\'nski's selection principle for normed spaces, Theorem \ref{mth:2}.

In the proof of Corollary \ref{mcor:1}, we use the arguments from the proof of Corollary 10.4.9 of \cite{bop}.

\begin{proof}[Proof of Corollary \ref{mcor:1}]
Consider the subspace $\overline{[y_{n}: n \in \mathbb{N}]}$ of $X$, that is separable. By Banach-Mazur's Theorem, $\overline{[y_{n}: n \in \mathbb{N}]}$ is isometrically isomorphic to a subspace of $C[0,1]$, so there exists a linear isometry $T:\overline{[y_{n}: n \in \mathbb{N}]} \to C[0,1]$ such that $\overline{[y_{n}: n \in \mathbb{N}]}$ and $T(\overline{[y_{n}: n \in \mathbb{N}]})$ are isometrically isomorphic. Observe that $\inf\limits_{n}\|T(y_{n})\|= \inf\limits_{n}\|y_{n}\| > 0$. Let $(x_{n})_{n=1}^{\infty}$ be a Schauder basis for $C[0,1]$, and $(x^{*}_{n})_{n}^{\infty}$ be the coefficient operators. By boundedness of operator $T$, since $(y_{n})_{n}$  converges to $0$ in the weak topology, we obtain that $\left( T(y_{n})\right)_{n}$ converges to $0$ in the weak topology. So

\begin{center}
$\lim\limits_{n\to \infty} x^{*}_{k}(Ty_{n})=0$ for all $k$ in $\mathbb{N}$.
\end{center}

By Bessaga-Pe\l{}czy\'nski's selection principle for normed spaces (Theorem \ref{mth:2}), there exists an essential basic subsequence $(T(y_{n_{k}}))_{k=1}^{\infty}$ of  $(T(y_{n}))_{n}^{\infty}$. Using that $T^{-1}: T(\overline{[y_{n}: n \in \mathbb{N}]}) \to \overline{[y_{n}: n \in \mathbb{N}]}$ is an isometric isomorphism, we have that $(y_{n_{k}})_{k=1}^{\infty}  = (T^{-1}(T(y_{n_{k}})))_{k=1}^{\infty}$ is an essential basic sequence. This concludes the proof of Corollary \ref{mcor:1}.
\end{proof}

\subsection{Proof of Banach problem for normed spaces}\label{sec:papp.2}

\begin{proof}[Proof of Theorem \ref{mthm:Banach-Mazur}]

We will use the same sequence  $(x_{n})_{n=1}^{\infty}$ obtained in the traditional proof (see e.g. \cite{daws,diestel}). By Banach-Grunblum's criterion for normed spaces, Theorem \ref{B-GforNormed}, we obtain that $(x_{n})_{n=1}^{\infty}$ is an essential basic sequence. In particular, $(x_{n})_{n=1}^{\infty}$ is an essential Schauder basis. Then, by Theorem (\ref{theorem11}), the canonical projections $(P_{n})_{n}$ are bounded and $\sup \limits_{n} \|P_{n}\| < + \infty $.
\end{proof}

\subsection{Proof of Spectral Theorem}\label{sec:papp.3}

We begin by proving a consequence of Theorem \ref{B-GforNormed}.

\begin{lemma}\label{lem.aux.111}
Let $N$ be an inner product space, $S=\{x_{n}: n \in \mathbb{N}\}$ be an orthonormal set of $N$. Then
\begin{itemize}
\item[$(i)$] For all $x \in N$, and any $n \in \mathbb{N}$, we have that $\sum\limits_{i=1}^{n} |\langle x,x_{i}\rangle|^2 \leq \|x\|^2$ (and then $\sum\limits_{i=1}^{\infty} |\langle x,x_{i}\rangle|^2 \leq \|x\|^2$);
\item[$(ii)$] $(x_{n})_{n=1}^{\infty}$ is an essential  Schauder basis for $\widehat{[x_{n}: n \in \mathbb{N}]} \subseteq \widehat{N}$;
\item[$(iii)$]  $(x_{n})_{n=1}^{\infty}$ is an essential Schauder basis for $\overline{[x_{n}: n \in \mathbb{N}]} \subseteq N$.
\end{itemize}

\end{lemma}
\begin{proof} Just note that  $0  \leq \langle x - \sum\limits_{i=1}^{n} \langle x,x_{i} \rangle x_{i}  , x - \sum\limits_{i=1}^{n} \langle x,x_{i} \rangle x_{i} \rangle$ to prove  item $(i)$.
To obtain items $(ii)$ and $(iii)$, by Banach-Grunblum's criterion for normed spaces, Theorem (\ref{B-GforNormed}), we are reduced to prove that for all  sequence of scalar $(a_{n})_{n=1}^{\infty}$ we have that $\left\Vert \sum \limits_{i=1}^{m} a_{i}x_{i}  \right\Vert \leqslant  \left\Vert \sum \limits_{i=1}^{n} a_{i}x_{i} \right\Vert$  whenever $n \geqslant m$. But it is clear by  orthonormality of $S$.
\end{proof}

We need these auxiliar results.

\begin{lemma}\label{spteL}
Let $N$ be  an inner product space and $T: N \to N$ be a compact self-adjoint operator. Then $\widehat{T}: \widehat{N} \to \widehat{N}$ is a compact self-adjoint operator where  $\widehat{T}$ is the bounded linear extension of $T$ and $\widehat{N}$ is the completion of $N$ such that $N$ is dense in $\widehat{N}$.
\end{lemma}
\begin{proof}
The proof is straightforward.
\end{proof}

\begin{lemma}\label{lem.aux.113}
Let $N$ be an inner product space and $U: N \to N$ be a compact self-adjoint operator. If $U$ is not the null-operator, there exist $x \in N \setminus \{0\}$ and $\lambda \in \mathbb{R}\setminus\{0\}$ such that $U(x)= \lambda x$.
\end{lemma}
\begin{proof}
Use the same arguments present in the proof of this result for compact self-adjoint operators on Hilbert spaces.
\end{proof}

\begin{proof}[Proof of Theorem \ref{sptns}] By Lemma \ref{spteL}, $\widehat{T}: \widehat{N} \to \widehat{N}$ is a  compact self-adjoint operator on a Hilbert space. There exist $\alpha_{i} \in \mathbb{R}\setminus \{0\}$ eigenvalues  of  $\widehat{T}$ for $i$ in a not empty subset $A$ of $\mathbb{N}$    such that $\ker(\widehat{T}- \alpha_{i}\widehat{I})$ is a finite dimensional subspace where $I: N \to N$ is given by $I(x) =x$ and  $\widehat{I}$ is the extension of $I$ to $\widehat{N}$ .  For each $n \in A$, note that   $N \cap \ker(\widehat{T}- \alpha_{n}\widehat{I}) =  \ker(T- \alpha_{n}I) \subseteq N$, define $t_{n} := \dim \ker(T- \alpha_{n}I)$ and $\widehat{t}_{n}:= \dim\ker(\widehat{T}- \alpha_{n}\widehat{I})$. Let be $J$ be the subset of $\mathbb{N}$ given by  $\{n \in A: N \cap \ker(\widehat{T}- \alpha_{n}\widehat{I}) \neq \{0\}  \}$, we have that $1 \leq t_{n} \leq \widehat{t}_{n}$. So for each $n\in J$ there exist $v_{n,1},\cdots,v_{n,t_{n}}$  such that  $T(v_{n,j}) = \alpha_{n} v_{n,j}$ for all $j \in \{1,\cdots, t_{n}\}$ and $\{v_{n,j}:  1 \leq j \leq t_{n}\}$ is an orthonormal set. Then $S= \{v_{n,j}: n \in J$ and $j \in\{1,\cdots, t_{n}\} \}$ is an orthonormal set of $N$. We may write the  orthonormal set $S$ as  $\{x_{\ell}  \in N:  \ell \in \mathcal{J}\}$ where $\mathcal{J}$ is a subset of $\mathbb{N}$, and $T(x_{\ell}) = \lambda_{\ell} x_{\ell}$ for each $\ell \in \mathcal{J}$.

For each $n \in A$, there exists  an orthonormal set  $\{v_{n,t_{n}+1},\cdots,v_{n,\widehat{t}_{n}}\}$ of   $\ker(\widehat{T}- \alpha_{n}\widehat{I}) $, and then $\{v_{n,1},\cdots,v_{n,\widehat{t}_{n}} \}$  is an orthonormal set of $\ker(\widehat{T}- \alpha_{n}\widehat{I}) $. Then $R= \{v_{n,j}: n \in A$ and $j \in\{t_{n}+1,\cdots, \widehat{t}_{n}\} \}$ is an orthonormal set of $\widehat{N}$. We may write the  orthonormal set $R$ as  $\{y_{\ell}  \in N:  \ell \in \widetilde{A}\}$ where $\widetilde{A}$ is a subset of $\mathbb{N}$, and $\widehat{T}(y_{\ell}) = \beta_{\ell} y_{\ell}$ for each $\ell \in \widetilde{A}$.

Suppose that $\spp (T)$ is an infinite set, then   $\mathcal{J} =  \mathbb{N}$. We are going to prove  items $(i)$ and $(ii)$.

 Let $\widehat{F}$ be the completion of $F=[x_{\ell}: \ell \in \mathcal{J}]$, $\widehat{G}$ be the completion of $G=[y_{\ell}: \ell \in \widetilde{A}]$, $\widehat{E}$ be the completion of $E=\left[ \bigcup\limits_{n=1}^{\infty}  \ker(\widehat{T}- \alpha_{n}\widehat{I})  \right]= [F \cup G]$. From the Spectral Theorem for Hilbert Spaces, we have that  $\widehat{N} = \widehat{E}  \oplus \ker (\widehat{T})$, and observe that $\widehat{E} = \widehat{F} \oplus \widehat{G}$.

 Since $S$ is an orthonormal set of $N$, from Lemma \ref{lem.aux.111}, we have that $(x_{\ell})_{\ell=1}^{\infty}$ is an essential  Schauder basis for $\widehat{F}$. For any $a \in \widehat{F}$,  there exists a  sequence of scalar $(\xi_{n})_{n=1}^{\infty}$ such that $a = \sum \limits_{i =1}^{\infty} \xi_{i}x_{i}$ where $\xi_{i} = \langle a, x_{i} \rangle$ for each $i \in \mathbb{N}$, so $\widehat{T}(a) =  \sum \limits_{i =1}^{\infty} x_{i} \lambda_{i} \langle a, x_{i} \rangle$.

  Since $R$ is an orthonormal set of $\widehat{N}$, from Lemma \ref{lem.aux.111}, we have that $(y_{\ell})_{\ell \in \widetilde{A}  }$ is an essential  Schauder basis for $\widehat{G}$. For any $b \in \widehat{G}$,  there exists a sequence of scalar $(\gamma_{n})_{n \in \widetilde{A}}$ such that $b = \sum \limits_{i\in \widetilde{A}} \gamma_{i}y_{i}$ where $\gamma_{i} = \langle b, y_{i} \rangle$ for each $i \in \widetilde{A}$, so $\widehat{T}(b) = \sum \limits_{i \in \widetilde{A}} y_{i} \beta_{i} \langle b, y_{i} \rangle$.

We need the following result.

\begin{lemma}\label{lem.aux.112}
$\widehat{T} (\widehat{F} ) \subseteq N \cap \widehat{F}$.
\end{lemma}
\begin{proof}[Proof of Lemma \ref{lem.aux.112}] Let $x$ be an arbitrary element of $\widehat{F}$.  Since $S$ is an orthonormal set of $N$, from Lemma \ref{lem.aux.111}, we have that $(x_{\ell})_{\ell=1}^{\infty}$ is an essential  Schauder basis for $\widehat{F}$. There exists a  sequence of scalar $(a_{n})_{n=1}^{\infty}$ such that $x = \sum \limits_{i=1}^{\infty} a_{i}x_{i}$ where $a_{i} = \langle x, x_{i} \rangle$ for each $i \in \widetilde{J}$.

Note that $\widehat{T}(x) = \sum \limits_{i=1}^{\infty} a_{i}\lambda_{i}x_{i} = \lim\limits_{n\to \infty}  \sum \limits_{i=1}^{n} a_{i}\lambda_{i}x_{i}$,  so $\widehat{T}(x)  \in \widehat{F}$.

Consider the sequence $(y_{n})_{n=1}^{\infty}$  where $y_{n}= \sum\limits_{i=1}^{n} x_{i}a_{i} \in N$ for each $n \in \mathbb{N}$. Using Lemma \ref{lem.aux.111}, we see that $\|y_{n}\|^2 = \sum\limits_{i=1}^{n} \left(a_{i} \right)^2  = \sum\limits_{i=1}^{n} |\langle x,x_{i}\rangle|^2 \leq \|x\|^2$. Then $(y_{n})_{n=1}^{\infty}$ is a bounded sequence of $N$, by compactness of $T$, there exists a subsequence $(T(y_{n_{k}}))_{n=1}^{\infty}$ that converges to some point $y \in N$. We have that $y_{n}$ converges to $x$, and then $\widehat{T}(y_{n})$ converges to $\widehat{T}(x)$, so $\widehat{T}(x) = y \in N$. This complets the proof of  Lemma \ref{lem.aux.112}.
\end{proof}

So for each $x \in \widehat{N}$ there exist $a_{x} \in  \widehat{F}$, $b_{x} \in  \widehat{G}$ and $c_{x} \in \ker (\widehat{T})$ such that $x = a_{x} + b_{x} + c_{x}$, and then $\widehat{T}(x) = \widehat{T}(a_{x}) +\widehat{T}(b_{x})$ where $\widehat{T}(a_{x}) =\sum \limits_{i=1}^{\infty} x_{i} \lambda_{i} \langle a_{x}, x_{i} \rangle$ and  $\widehat{T}(b_{x}) = \sum \limits_{i \in \widetilde{A}} y_{i} \beta_{i} \langle b_{x}, y_{i} \rangle$. By  Lemma \ref{lem.aux.112}, we get that $\widehat{T}(a_{x})  \in N$.

Now, suppose that $x \in N$, so $T(x)= \widehat{T}(x) \in N$ and $\widehat{T}(x) = \widehat{T}(a_{x}) +\widehat{T}(b_{x})$ with $\widehat{T}(a_{x}) \in N$,  then $\widehat{T}(b_{x}) \in N$. Note that  $\langle a_{x}, x_{i} \rangle = \langle x, x_{i} \rangle$ and $\langle b_{x}, y_{i} \rangle = \langle x, y_{i} \rangle$. For $x \in N$, we obtain that $T(x) =  \sum \limits_{i=1}^{\infty} x_{i} \lambda_{i} \langle x, x_{i} \rangle + \sum \limits_{i \in \widetilde{A}} y_{i} \beta_{i} \langle x, y_{i} \rangle$ where  $\sum \limits_{i=1}^{\infty} x_{i} \lambda_{i} \langle x, x_{i} \rangle, \sum \limits_{i \in \widetilde{A}} y_{i} \beta_{i} \langle x, y_{i} \rangle \in N$. This proves item $(i)$.

Suppose that $\spp (T) =\spp (\widehat{T})$, this implies that $\widetilde{A} = \emptyset$, and then $T(x) =  \sum \limits_{i=1}^{\infty} x_{i} \lambda_{i} \langle x, x_{i} \rangle$ for all $x \in N$. The item $(ii)$ is proved.

It remains to show item $(iii)$. Now, $\spp (T)$  is a natural number $k$. This implies that $S= \{v_{n,j}: n \in J$ and $j \in\{1,\cdots, t_{n}\} \}$ is a finite orthonormal set of $N$ such that $T(v_{n,j}) = \alpha_{n} v_{n,j}$ for all $j \in \{1,\cdots, t_{n}\}$ for each $n \in J$.  We may write the  orthonormal set $S$ as  $\{x_{\ell}  \in N:  \ell \in \{1,\cdots,s\}\}$ for some $s \in \mathbb{N} \setminus \{0\}$ where $T(x_{\ell}) = \lambda_{\ell} x_{\ell}$ for each $\ell \in \{1,\cdots, s\}$.

Note that  $F= [x_{\ell}: \ell \in \{1,\cdots, s\} ] $ is a Banach space since $F$ is a finite dimensional space. This implies that $N = F \oplus F^{\perp}$  where  $F^{\perp} = \{w \in N: \langle w,z \rangle =0$ for all $z \in F \}$.

We claim that $F^{\perp} = \ker (T)$. A trivial verification shows that $\ker (T) \subseteq F^{\perp}$. Using that $T(F) \subseteq F$, we have that $T(F^{\perp}) \subseteq F^{\perp}$. Define $U=T|_{F^{\perp}}: F^{\perp} \to F^{\perp}$, and note that $U$ is a compact self-adjoint operator. Suppose that $F^{\perp}$ is not contained in $\ker (T)$, there exists $q \in F^{\perp}$ such that $T(q) \neq 0$, then $U(q) = T(q) \neq 0$ and $U$ is not null-operator. By Lemma \ref{lem.aux.113}, there exist $\lambda \in \mathbb{R}$  and $p \in F^{\perp} \setminus \{0\}$ such that $T(p)=U(p) = \lambda p$, so $p \in F$. We obtain that $p \in F \cap F^{\perp}$, so $p =0$. This contradiction shows that $U$ is the null operator, and the claim is proved.

For $x \in N = F \oplus \ker (T)$, there exist $a_{x} \in F $ and $b_{x} \in \ker (T)$  such that $x = a_{x} + b_{x}$. Note that for $a \in F$, we have that $a = \sum \limits_{i=1}^{s}  \langle a,x_{i} \rangle x_{i}$ and $T(a) = \sum \limits_{i=1}^{s}  \langle a,x_{i} \rangle \lambda_{i} x_{i}$. We obtain that $T(x) = T(a_{x}) =  \sum \limits_{i=1}^{s}  \langle x,x_{i} \rangle \lambda_{i} x_{i}$ since  $\langle a_{x},x_{i} \rangle = \langle x,x_{i} \rangle$ for all $i \in \{1,\cdots,s\}$, and we are done.
\end{proof}

\section{Proofs of main results}\label{sec:main.result}
\subsection{The Banach-Grunblum's criterion for normed spaces}\label{sec:main.result.A}

\begin{proof}[Proof of Theorem \ref{B-GforNormed}] It is clear that $(i)$ implies  $(ii)$. We are going to show that $(ii)$ proves  $(iii)$.

Suppose that  $(x_{n})_{n=1}^{\infty}$ is an essential Schauder basis for $\overline{[x_{n}: n \in \mathbb{N}]}$. Now, consider the canonical projections $(P_{n})_{n=1}^{\infty}$ in $E=\overline{[x_{n}: n \in \mathbb{N}]}$. By Theorem (\ref{theorem11}), we know that
\begin{center}
$1 \leqslant   K_{(x_{n})_{n=1}^{\infty}} = \sup \limits_{n} \| P_{n}\| < \infty$
\end{center}

Given a   sequence of scalar $(a_{n})_{n=1}^{\infty}$, if $n \geqslant m$, then

\begin{align} \label{eq.653}
\nonumber \Vert \sum \limits_{i=1}^{m} a_{i}x_{i} \Vert\\ \nonumber = \Vert P_{m}(\sum \limits_{i=1}^{n} a_{i}x_{i})\Vert \\  \leqslant \|P_{m}\|\cdot \Vert \sum \limits_{i=1}^{n} a_{i}x_{i} \Vert\\ \nonumber \leqslant K_{(x_{n})_{n=1}^{\infty}} \Vert \sum \limits_{i=1}^{n} a_{i}x_{i} \Vert.
\end{align}

We are reduced to proving $(i)$ from $(iii)$. Suppose that (\ref{eq12}) holds for $M \geqslant 1$. A straightforward calculation shows that the set $\{x_{\ell}: \ell \in \mathbb{N}\}$   is linearly independent.

For each $n \in \mathbb{N}$ consider the linear functional given by

\begin{center}
$\varphi_{n}: [x_{\ell}: \ell \in \mathbb{N}] \to \mathbb{K}\text{, }\varphi_{n}\left(\sum \limits_{i=1}^{k} a_{i}x_{i}\right)=a_{n}$
\end{center}
and the linear operator

\begin{center}
$T_{n}: [x_{\ell}: \ell \in \mathbb{N}] \to [x_{\ell}: \ell \in \mathbb{N}]\text{, }T_{n}\left(\sum \limits_{i=1}^{k} a_{i}x_{i}\right)=\sum \limits_{i=1}^{n} a_{i}x_{i}$
\end{center}
defining $a_{k+1}=\cdots=a_{n}=0$ if necessary.

It is clear that $\varphi_{n}$ is a bounded linear operator for all $n$.

By (\ref{eq12}), we have that

\begin{center}
$\left\Vert T_{n} \left( \sum \limits_{i=1}^{k} a_{i}x_{i} \right) \right\Vert = \left\Vert \sum \limits_{i=1}^{n} a_{i}x_{i} \right\Vert \leq M \left\Vert \sum \limits_{i=1}^{k} a_{i}x_{i} \right\Vert$
\end{center}
and $T_{n}$ is bounded with $\|T_{n}\| \leq M$.

There exists a bounded linear extension $\Phi_{n}: \widehat{F}\to \mathbb{K}$ such that $ \Phi_{n}|_{F}=\varphi_{n}$ and $ \|\Phi_{n} \|= \| \varphi_{n} \|$ where $F=[x_{\ell}: \ell \in \mathbb{N}]$.

Consider the bounded linear operator $T_{n}: F \to F$, so we may consider this bounded linear operator $T_{n}: [x_{\ell}: \ell \in \mathbb{N}] \to \widehat{F}$. There exists a bounded linear extension $R_{n}: \widehat{F}\to \widehat{F} $ such that  $ R_{n}|_{F}=T_{n}$ and $ \| R_{n} \|= \| T_{n} \|$.

Note that $T_{n}(z) = \sum \limits_{i=1}^{n} a_{i}x_{i} = \sum \limits_{i=1}^{n} \varphi_{i}(z)x_{i} $ for all $z \in F$.

For all $x \in \widehat{F}$ we have that

\begin{align}\label{eq156}
R_{n}(x) =  \sum \limits_{i=1}^{n} \Phi_{i}(x)x_{i}.
\end{align}

In fact, let $x\in \widehat{F} $ be an arbitrary element, so $x= \lim \limits_{k\to\infty} y_{k}$ where $y_{k} \in [x_{\ell}: \ell \in \mathbb{N}]$. But $ R_{n}$ is a bounded linear operator, so

\begin{center}
$R_{n}(x) =   \lim \limits_{k\to\infty} R_{n}(y_{k}) =  \lim \limits_{k\to \infty}\sum \limits_{i=1}^{n} \varphi_{i}(y_{k})x_{i} = \lim \limits_{k\to \infty}\sum \limits_{i=1}^{n} \Phi_{i}(y_{k})x_{i} =  \sum \limits_{i=1}^{n} \Phi_{i}(x)x_{i} $.
\end{center}

We obtain that $R_{n}(x) = \sum \limits_{i=1}^{n} \Phi_{i}(x)x_{i}$ for all $x$ in  $ \widehat{F}$. Now, given $x \in \widehat{F} $ and $\varepsilon > 0$, there exists $y = \sum \limits_{j=1}^{m} a_{j}x_{j} \in [x_{\ell}: \ell \in \mathbb{N}]$ for some $m \geqslant 1  $ such that $\| x-y \| < \varepsilon$. For $n > m$, we have

\begin{align*}
\|x- R_{n}(x) \| \\ \leqslant \| x-y \| + \| R_{n}(y)-y \| \\ + \| R_{n}(x)-R_{n}(y) \| 
\\ \leqslant \| x-y \| + \| y -y \| \\+\|R_{n}\|\cdot \| x- y \| \leq (1+M)\varepsilon.
\end{align*}

Then $x= \lim \limits_{n \to \infty} R_{n}(x)$. Using (\ref{eq156}), we obtain

\begin{center}
$x= \lim \limits_{n \to \infty} R_{n}(x) = \lim \limits_{n \to \infty} \sum \limits_{i=1}^{n} \Phi_{i}(x)x_{i} = \sum \limits_{i=1}^{\infty} \Phi_{i}(x)x_{i}$.
\end{center}

The uniqueness of the above representation is clear. So $(x_{n})_{n=1}^{\infty}$ is a Schauder basis for $\widehat{F} $.

Given $x \in \widehat{F} $ with $x =  \sum \limits_{i=1}^{\infty} a_{i}x_{i}$. By (\ref{eq12}), for each $n \in \mathbb{N}$ we have that

\begin{align}
\left\Vert \sum \limits_{i=1}^{n} a_{i}x_{i} \right\Vert \leq M \left\Vert \sum \limits_{i=1}^{\infty} a_{i}x_{i} \right\Vert = M \Vert x\Vert.
\end{align}

Then

\begin{align}\label{eq13}
\sup \limits_{n} \Vert \sum \limits_{i=1}^{n} a_{i}x_{i} \Vert \leq M \Vert \sum \limits_{i=1}^{\infty} a_{i}x_{i} \Vert = M \Vert x\Vert.
\end{align}

and note that

\begin{align}\label{eq14}
\nonumber \| x \| =  \left \Vert  \sum \limits_{i=1}^{\infty} a_{i}x_{i} \right \Vert \\ = \lim \limits_{n\to \infty} \left \Vert \sum \limits_{i=1}^{n} a_{i}x_{i} \right \Vert \leq \sup \limits_{n} \left \Vert \sum \limits_{i=1}^{n} a_{i}x_{i} \right \Vert.
\end{align}

Using (\ref{eq13}) and  (\ref{eq14}), we obtain that $T_{\widehat{F}}: \mathcal{L}_{\widehat{F}} \to \widehat{F}$ given by $T_{\widehat{F}}((a_{n})_{n=1}^{\infty}) = \sum \limits_{n=1}^{\infty} a_{n}x_{n}$ is a linear isomorphism. Then $(x_{n})_{n=1}^{\infty}$ is an essential Schauder basis for $\widehat{F}$, and we are done.
\end{proof}

\subsection{The Bessaga-Pe\l{}czy\'nski Theorem for normed spaces}\label{sec:main.result.B}

\begin{proof}[Proof of Theorem \ref{th:1}]
Given a sequence $(a_{n})_{n=1}^{\infty}$  of scalars,

\begin{align*}
\Vert\sum \limits_{i=1}^{n} a_{i}(x_{i}-y_{i}) \Vert \\= \Vert\sum \limits_{i=1}^{n} x_{i}^{*}(\sum \limits_{j=1}^{n} a_{j}x_{j})(x_{i}-y_{i}) \Vert \\ \leqslant \sum \limits_{i=1}^{n} \vert x_{i}^{*}(\sum \limits_{j=1}^{n} a_{j}x_{j})\vert.\| x_{i}-y_{i} \| \\
\leqslant \Vert \sum \limits_{j=1}^{n} a_{j}x_{j}  \Vert (\sum \limits_{i=1}^{n} \Vert x_{i}^{*}\Vert.\Vert x_{i}-y_{i} \Vert)\\ \leqslant \lambda \Vert \sum \limits_{i=1}^{n} a_{i}x_{i}  \Vert,
\end{align*}
we obtain that

\begin{center}
$\left\vert \left\Vert \sum \limits_{i=1}^{n} a_{i}x_{i} \right\Vert - \left\Vert \sum \limits_{i=1}^{n}a_{i}y_{i} \right\Vert \right\vert \leqslant \left\Vert\sum \limits_{i=1}^{n} a_{i}x_{i}-  \sum \limits_{i=1}^{n}a_{i}y_{i} \right\Vert \leqslant \lambda \left\Vert \sum \limits_{i=1}^{n} a_{i}x_{i}  \right\Vert$.
\end{center}

So
\begin{align} \label{eq10.15}
\nonumber (1-\lambda) \Vert \sum \limits_{i=1}^{n} a_{i}x_{i} \Vert\\ \leqslant  \Vert \sum \limits_{i=1}^{n} a_{i}y_{i} \Vert \\ \nonumber \leqslant  (1+\lambda) \Vert\sum \limits_{i=1}^{n} a_{i}x_{i} \Vert,
\end{align}
for all $n$ in $\mathbb{N}$. But  $(x_{n})_{n=1}^{\infty}$ is an essential basic sequence in $X$. By Banach-Grunblum's criterion for normed spaces, Theorem (\ref{B-GforNormed}), there exists $M \geqslant 1$ such that if $m \geqslant n$ then

\begin{align} \label{eq10.16}
\nonumber \Vert\sum \limits_{i=1}^{n} a_{i}y_{i} \Vert \\ \leqslant  (1+\lambda) \Vert\sum \limits_{i=1}^{n} a_{i}x_{i} \Vert\\ \nonumber \leqslant (1+\lambda)M \Vert \sum \limits_{i=1}^{m} a_{i}x_{i} \Vert \\ \nonumber \leqslant \frac{(1+\lambda)M}{(1-\lambda)} \Vert\sum \limits_{i=1}^{m} a_{i}y_{i} \Vert.
\end{align}

From Banach-Grunblum's criterion for normed spaces, Theorem (\ref{B-GforNormed}), $(y_{n})_{n=1}^{\infty}$ is an essential basic sequence in $X$.

We are going to show that the series $\sum \limits_{n=1}^{\infty} a_{n}y_{n}$ converges if  the series $\sum \limits_{n=1}^{\infty} a_{n}x_{n}$ is convergent.

Define $F= [x_{n}: n \in \mathbb{N}]$  and  $Y= \overline{[y_{n}: n \in \mathbb{N}]}$. Consider $T:F \to Y$ the linear operator given by $T \left( \sum\limits_{i=0}^{k} x_{i}a_{i} \right) = \sum\limits_{i=0}^{k} y_{i}a_{i}$ for $x = \sum\limits_{i=0}^{k} x_{i}a_{i}$ in $F$. Note that if $x \neq 0$, we have that $\left\Vert  T \left( \frac{x}{\|x\|} \right) \right\Vert = \frac{\left\Vert T(x) \right\Vert}{\left\Vert x \right\Vert} = \frac{ \left\Vert\sum \limits_{i=1}^{k} a_{i}y_{i} \right\Vert}{ \left\Vert \sum \limits_{i=1}^{k} a_{i}x_{i} \right\Vert} $ for $x = \sum\limits_{i=0}^{k} x_{i}a_{i}$ in $F$, and by (\ref{eq10.15}) we obtain that $\left\Vert T \left( \frac{x}{\|x\|} \right) \right\Vert \leqslant   (1+\lambda)$.

So $T:F \to Y$ is a bounded linear operator. We may consider $T$ as the following bounded operator $T: \overline{F} \subseteq X \to Y$. There exists a bounded linear extension $\widehat{T}:\widehat{F} \to\widehat{Y}$ such that  $\widehat{T}|_{\overline{[x_{n}: n \in \mathbb{N}]}}  = T$ with $\|\widehat{T} \| = \|T\|$.

Now, suppose that $x=\sum \limits_{n=1}^{\infty} a_{n}x_{n} \in \overline{F} \subseteq X$ is convergent. This implies that

\begin{center}
$T(x)=\widehat{T}(x)=  \widehat{T}\left(\lim\limits_{n \to \infty} \sum\limits_{i=1}^{n} a_{i}x_{i}\right)  =  \lim\limits_{n \to \infty}   \widehat{T} \left( \sum\limits_{i=1}^{n} a_{i}x_{i} \right)  = \lim\limits_{n \to \infty}   \sum\limits_{i=1}^{n} a_{i}y_{i}$,
\end{center}

then the series $T(x)=\widehat{T}(x) = \sum \limits_{i=1}^{\infty} a_{i}y_{i}\in Y \subseteq X $ is convergent.

Analogously, we can show that if   the $\sum \limits_{i=1}^{\infty} a_{i}y_{i}$ is convergent, then $\sum \limits_{i=1}^{\infty} a_{i}x_{i}$ is convergent.
\end{proof}

\section{Conclusion}

We give a generalization of the well known criterion of Banach-Grunblum and the Bessaga-Pe\l{}czy\'nski Theorem in normed spaces setting (not necessarily Banach one). As application, we show the Principle of Selection of Bessaga-Pe\l{}czy\'nski for normed spaces and the Spectral Theorem for compact self-adjoint operators on inner product spaces. 

In the case of Banach spaces (normed and complete), this theory can be applied to solve differential equations with a low computational cost \cite{PR2005}. This is an important issue in applied sciences, because many problems can be formulated through differential/integral equations. Although our work in this paper is not specifically on applied science, we expect that our generalization of this theory to spaces beyond Banach ones help to improve its application. Also, in  it has been developed a method, by using the least-squares method and a Schauder basis \cite{PPRR2008}, which provides a numerical solution for a wide class of linear differential or integral equations. The least squares method is a statistical technique used to find the line of best fit for a given set of data points. This method minimizes the sum of the squares of the residuals, which are the differences between the observed values and the values predicted by the model. We hope that our techniques help to improve these results and be useful to applied sciences in some moment.

\flushleft \par{\textit {Acknowledgment:}}\\
No fee has been paid for this publication. 

The authors are grateful to Instituto de Matem\'atica da Universidade Federal da Bahia for the hospitality in preparation of this paper.


\def\cprime{$'$}


\begin{thebibliography}{9}

\bibitem{AlKa06}
F. ~Albiac. N.J. ~Kalton.
\newblock Topics in Banach Space Theory.
\newblock {\em Graduate Texts in Mathematics, 233.}
\newblock  New York: Spring-Verlag, 2006.

\bibitem{Banach32}
S. Banach.
\newblock{Th\'eorie des op\'erations lin\'eaires.}
\newblock{Monografie Matematyczne 1.}
\newblock Warsaw, 1932.

\bibitem{bespel58}
C. ~Bessaga; A. ~Pe\l{}czy\'nski.
\newblock On bases and unconditional convergence of series in Banach spaces.
\newblock {\em Studia Mathematica (1958).}
\newblock  Volume: 17, Issue: 2, page 151-164.



\bibitem{bop}
G. ~Botelho, D. ~Pellegrino, E. ~Teixeira.
\newblock Fundamentos de analise funcional.
\newblock {\em Colecao textos universitarios.}
\newblock Rio de Janeiro: SBM, 2015.



\bibitem{Sw10}
C. ~Swartz.
\newblock Elementary Functional Analysis.
\newblock {\em  World Scientific.}
\newblock Singapore, 2010.




\bibitem{Day62}
M.M. ~Day.
\newblock On the Basis Problem in Normed Spaces.
\newblock {\em Proceedings of the American Mathematical Society.}
\newblock Vol. 13, No. 4 (Aug., 1962), pp. 655-658.



\bibitem{daws}
M. ~Daws.
{Introduction} to Bases in Banach Spaces, 2005.

www1.maths.leeds.ac.uk/~mdaws/\\pubs/bases.pdf

\bibitem{diestel}
J. ~Diestel.
\newblock Sequences and Series in Banach Spaces.
\newblock {\em Graduate texts in mathematics; 92}
\newblock  New York: Spring-Verlag, 1984.

\bibitem{enflo}
P. ~Enflo.
\newblock A counterexample to the approximation problem in Banach spaces. \newblock {\em Acta Mathematica.}
\newblock  130 (1): 309-317, 1973.

\bibitem{G58}
B. R. Gelbaum.
\newblock Notes on Banach spaces and bases.
\newblock {\em An. Acad. Bras. Cienc.}
\newblock 30 (1958), 29-36.

\bibitem{GoMa}
W.T. ~Gowers, B.~Maurey.
\newblock The unconditional basic sequence problem.
\newblock arXiv:math/9205204, 1992.


\bibitem{Mc72}
C. W. ~McArthur.
\newblock Developments in Schauder basis theory.
\newblock Bull. Amer. Math. Soc. 78 (1972), 877-908.


\bibitem{Me98}
R.E. ~Megginson.
\newblock An introduction do Banach space theory.
\newblock Springer, 1998.

\bibitem{PPRR2008}
A.~Palomares, M.~Pasadas, V.~Ram\'irez, M.~Ruiz Gal\'an.
\newblock A convergence result for a least-squares method using Schauder bases.
\newblock Mathematics and Computers in Simulation 77 (2008) 274--28.

\bibitem{PR2005}
A.~Palomares, M.~Ruiz Gal\'an.
\newblock Isomorphisms, Schauder bases in Banach spaces and numerical solution of integral and differential equations.
\newblock Numerical Functional Analysis and Optimization, 26(1):129--137, 2005.

\bibitem{Sch1927}
J.~Schauder.
\newblock Theorie stetiger Abbildungen in Funktionalr\"umen.
\newblock Math. Z., 26, 47-65, 1927.

\bibitem{Sch1928}
J.~Schauder.
\newblock Eine eigenschaft des haarschen orthogonalsystems.
\newblock Math. Z., 28, 317-320, 1928.

\bibitem{VK2022}
D.~Varsamis, A.~Kamilali.
\newblock Calculation of Determinant of a two-variable Polynomial Matrix in Complex Basis.
\newblock WSEAS Transactions on systems and control, (2022) Vol.17. DOI: 10.37394/23203.2022.17.44

\bibitem{Yar2022}
M.~Yaremenko.
\newblock Trace class in separable reflexive Banach spaces, Lidskii theorem.
\newblock Equations, an International Journal of Mathematical and Computational Methods in Science and Engineering, vol.2, 19, 2022. E-ISSN: 2732--9976 DOI:10.37394/232021.2022.2.19 
\end{thebibliography}

\end{document}